\documentclass[11pt]{article}

\usepackage[utf8]{inputenc}
\usepackage{amssymb, amsmath, amsthm}
\usepackage{enumerate}
\usepackage{hyperref}
\usepackage{tikz-cd}
\usepackage{todonotes}
\usepackage{xcolor}

\newtheorem{introtheorem}{Theorem}

\newtheorem{theorem}{Theorem}[section]
\newtheorem{lemma}[theorem]{Lemma}
\newtheorem{conjecture}[theorem]{Conjecture}

\newtheorem{proposition}[theorem]{Proposition}
\newtheorem{corollary}[theorem]{Corollary}

\theoremstyle{definition}
\newtheorem{definition}[theorem]{Definition}

\theoremstyle{remark}
\newtheorem{remark}[theorem]{Remark}
\newtheorem{question}[theorem]{Question}

\newcommand{\bbP}{\ensuremath{\mathbb{P}}}

\newcommand{\bbC}{\ensuremath{\mathbb{C}}}

\newcommand{\bbR}{\ensuremath{\mathbb{R}}}

\newcommand{\cO}{\ensuremath{\mathcal{O}}}
\newcommand{\BP}{\ensuremath{\mathcal{BP}}}

\newcommand{\Sing}{\ensuremath{\operatorname {Sing}}}
\newcommand{\WSing}{\ensuremath{\operatorname {{\mathcal{M}}ult}}}
\newcommand{\MSing}{\ensuremath{\operatorname {Mult}}}

\newcommand{\mult}{\ensuremath{\operatorname {mult}}}
\newcommand{\PGL}{\ensuremath{\operatorname {PGL}}}

\newcommand{\Orb}{\ensuremath{\mathbf{O}}}

\title{Harbourne constants, pull-back clusters and ramified morphisms}
\author{Piotr Pokora and Joaquim Roé}
\date{}

\begin{document}
\maketitle

\begin{abstract}
	We describe the effect of ramified morphisms on Harbourne constants of reduced effective divisors. With this goal, we introduce the pullback of a weighted cluster of infinitely near points under a dominant morphism between surfaces, and describe some of its basic properties. As an application, we describe configurations of curves with transversal intersections and $H$-index arbitarily close to \(-25/7\simeq -3.571\), smaller than any previously known result.
\end{abstract}

\section{Introduction}

The question whether, in every algebraic surface,  self-intersections of irreducible and reduced curves are bounded from below has intrigued algebraic geometers for decades, and continues to do so. The so-called Bounded Negativity Conjecture (BNC for short) is an old \emph{folklore conjecture}, now formally posed by Bauer et al. in \cite{BHKKMRS} (where some of its history is also explained), that asserts an affirmative answer:

\begin{conjecture}[BNC]\label{bnc}
Let $S$ be a smooth complex projective surface. Then there exists a positive integer $b(S) \in \mathbb{Z}$ such that for every irreducible and reduced curve $C \subset S$ one has $C^{2} \geq -b(S)$.
\end{conjecture}

By \emph{curve} in this paper we mean an effective (reduced) divisor on $S$.

In recent years the question of bounded negativity has received considerable attention, especially via the approach of trying to determine classes of surfaces $S$ which satisfy Conjecture \ref{bnc}
(see \cite{BHKKMRS, Har10, MT15, Rou17}).
In particular, \cite[Problem 1.2]{BdRHHLPSz} raised the question whether bounded negativity is a birational property, which leads to the following question:

\begin{question}
	Let $S$ be a smooth complex projective surface, and assume that $b(S) \in \mathbb{Z}$ is a positive integer such that for every irreducible and reduced curve $C \subset S$ one has $C^{2} \geq -b(S)$. 
	Let $n \ge 1$ be an integer.
	Is there a positive integer $b(S,n) \in \mathbb{Z}$ such that for every morphism $S_\pi\overset{\pi}{\rightarrow}S$ which is the composition of $n$ point blowups, and every irreducible curve $C \subset S_\pi$, one has $C^{2} \geq -b(S,n)$?
\end{question} 

If the answer to this question were positive for a given surface $S$ and every $n\ge 1$, then all smooth projective surfaces birational to $S$ would satisfy the BNC, but there is not a single surface on which the answer is known for all $n$.
Even for the simplest case of $S=\bbP^2$, the existence of $b(\bbP^2,n)$ is unknown for all $n>9$. 

Since the self-intersection of the strict transform $\tilde{C}$ by the blowup $\pi_p$ of $S$ at a point $p\in C$ is $\tilde{C}^2=C^2-\mult_{p}^{2}C$, we expect $b(S,n)$, if it exists, to be an increasing function with respect to $n$. 
In the case of the plane, the existence of rational nodal curves of every degree shows that if $b(\bbP^2,n)$ exists for all $n$, then $\liminf_{n\to\infty} (b(\bbP^2,n)/n)\ge 2$
(indeed, a rational nodal curve of degree $d$ has exactly $n=(d-1)(d-2)/2$ nodes, and after blowing up these points its strict transform has self-intersection $d^2-4n\simeq-2n+3\sqrt{2n}$).
In particular $b(\bbP^2,n)$ must grow at least linearly with $n$. 

No sequence of irreducible curves with negativity growing faster than $2n$ is known, but it was shown in \cite[Proposition 3.8.2]{RecDev} (see also \cite[Proposition 5.1]{BHKKMRS}) that the existence of a bound like $b(S,n)$ for prime or merely reduced divisors is equivalent.
This has led to the search of new examples considering possibly \emph{reducible reduced curves}. Define
\[
h(S,n)= \inf_{\substack{S_\pi \rightarrow S\\ n\text{-pt blowup}}}
\left\{\inf_{\substack{{C\subset S_\pi} \\ \text{reduced}}} \frac{C^2}{n}\right\} 
\in \mathbb{R} \cup \{-\infty\},
\]
so that $b(S,n)$ exists if and only if $h(S,n)$ is finite. 
Examples show that $\liminf_{n\to\infty} h(\bbP^2,n) \le -4$ (see \cite{Rou17}) and no sequence of examples has been found with larger than linear growth, so \cite[Problem 3.10]{BdRHHLPSz} asks whether in fact $\lim_{n\to\infty}h(\bbP^2,n)= -4$. 

A consequence of our work is that, even if $\inf_n h(\bbP^2,n)$ were finite (which remains unknown), it would not be equal to $C^2/n$ for any curve in a blowup $S_\pi \rightarrow \bbP^2$ of $\bbP^2$ at $n$ points:

\begin{introtheorem}\label{intro:strict}
	Let $h=\inf_{n}h(\bbP^2,n)$. For every morphism $S_\pi\overset{\pi}{\rightarrow}\bbP^2$ which is an $n$ point blowup, and every reduced curve $C \subset S_\pi$, one has $C^{2} >h\cdot n$.
\end{introtheorem}

Considerations like above led to the introduction of $H$-constants and $H$-indices \footnote{The $H$ in the name of the invariants may refer to the Hades, or underworld, of unknown negative curves, or to Brian Harbourne, one of the main contributors in these developments.} for reduced curves on smooth projective surfaces, with special emphasis on the plane case \cite{BdRHHLPSz}. 
These indices can be viewed as the average intersection numbers of negative curves by the number of singular points that they possess. 

\begin{definition}\label{def:H-constant-ordinary}
	Let $C \subset \mathbb{P}^{2}$ be a reduced curve of degree $d$, and let $K\subset\mathbb{P}^{2}$ be a finite set. The Harbourne constant of $C$ at $K$ is defined as
	$$H(C,K) = \frac{d^{2} - \sum_{p\in K} \mult_{p}(C)^{2}}{|K|},$$
	where $|K|$ denotes the cardinality of $K$. 
	
	The Harbourne index of a curve $C$ with ordinary singularities (a curve singularity is ordinary if it consists of smooth branches meeting transversely) is the Harbourne constant of $C$ at the set of singular points:
	\[h(C)=H(C,\Sing(C)).\]
\end{definition}

The most negative Harbourne index for curves with ordinary singularities found so far in the literature is provided by Wiman's configuration of lines $W$ \cite{BdRHHLPSz}, which has $h(W)=-225/67 \simeq -3.358$.  
In this work we provide more negative examples: 

\begin{introtheorem}\label{intro:newrecord}
	There exist reduced curves $C\subset \bbP^2$ with ordinary singularities and Harbourne indices $h(C)$ arbitrarily close to
\(-{25}/{7}\simeq -3.571.\) 
\end{introtheorem}

In sharp contrast with all previously known examples with very negative Harbourne index, the curves in Theorem \ref{intro:newrecord} do not have a large stabilizer group in $\PGL_3(\bbC)$.

For curves with non-ordinary singularities (such as the examples mentioned above showing $\liminf_{n\to\infty}h(\mathbb{P}^2,n)\le -4$), it is natural to modify the definition of Harbourne constants and indices by allowing some of the points in $K$ to be \emph{infinitely near}. In section \ref{sec:clusters} we introduce the notion of Harbourne constant at a \emph{multi-cluster of infinitely near points} and extend the definition of Harbourne index to arbitrary curves on smooth surfaces.

In order to prove Theorems \ref{intro:strict} and \ref{intro:newrecord}, we study pullbacks of suitable curves by ramified morphisms; in fact, the effect of ramified morphisms on $H$-constants is the main theme of this work.
Our motivation for this study stems from \cite{PR}, where we observed that the pullback of a reduced curve $C\subset \bbP^2$ by a ramified morphism $\bbP^2\rightarrow\bbP^2$ may have a more negative $H$-index than the original curve $C$.
Even if one is primarily interested in curves with ordinary singularities, their pullbacks by ramified morphisms may acquire non-ordinary singularities; to understand these, we apply the methods of \cite{Cas07}. In particular, clusters of infinitely near points and the corresponding extension of Definition \ref{def:H-constant-ordinary} become essential tools.

Let us stress that the idea to use pullback curves is rather natural in the context of negative curves. 
For instance, in positive characteristic it leads to a well-known counterexample to the BNC -- using the powers of the Frobenius endomorphism on the product $X = C \times C$, where $C$ is a genus $g(C) \geq 2$ curve, one can create an unbounded negativity phenomenon. 
In sharp contrast, T.~Bauer et~al.\ proved in  \cite{BHKKMRS} that over the complex numbers every surface admitting a surjective endomorphism which is not an isomorphism has bounded negativity. 
We expect that Theorem \ref{intro:strict} actually holds on every smooth projective surface admitting a ramified endomorphism; such a surface $S$ must have $\kappa(S)=-\infty$ by \cite[Lemma 2.3]{Fuj02}.

Given a surjective morphism $f:S\rightarrow S'$ of surfaces and a set ${K}$ of proper and infinitely near points (more precisely, a multi-cluster, see section \ref{sec:clusters}) on $S'$ with assigned multiplicities, we define in section \ref{sec:pullback_cluster} a \emph{pull-back multi-cluster} $f^*({K})$ with multiplicities, such that for every curve $C$ going through the points of ${K}$ with the assigned multiplicities, $f^*(C)$ goes through $f^*({K})$ with the pullback multiplicities.
If $f$ does not contract any curve to a point, we can control the number of points in $f^*({K})$ and their multiplicities using the local multiplicity $\nu_{p}(f)$ of $f$ at each proper point $p\in f^*({K})$ (writing $f$ in local coordinates as a pair of power series, $\nu_p(f)$ is the minimum of the orders of both power series, see section \ref{sec:pullback_cluster} or \cite{Cas07}). 
We obtain the following (cf. \cite[Lemma 7]{Rou17}).

\begin{introtheorem}\label{intro:pullback-H}
	Let $f:S \rightarrow S'$ be a finite morphism of smooth projective surfaces, $C\subset S'$ a reduced curve, and $K$ a multi-cluster on $S'$. Assume that $f^*(C)$ is reduced and $H(C,K)\le 0$. Then
	\[H(f^*(C),f^*(K))\le H(C,K),\]
	with a strict inequality if there is a point $p \in f^{*}({K})$ with $\nu_p (f)>1$.
\end{introtheorem}

The notion of pullback cluster and Theorem \ref{intro:pullback-H} form the technical core of this paper. 
Because of their generality we expect that their application will not be restricted to the bounded negativity conjecture, so these might be of independent interest.

\section{Preliminaries}

By a \emph{surface} $S$ we mean a connected 2-dimensional complex (analytic) manifold (so, smooth and irreducible). 
Unless otherwise stated we always work with the analytic topology.

A bimeromorphic map between surfaces is a proper holomorphic map $\pi:S_\pi\rightarrow S$ such that there exist proper analytic subsets $T\subset S$ and $T'\subset S_\pi$ such that $\pi$ restricts to an isomorphism $S_\pi \setminus T'\rightarrow S\setminus T$.
A bimeromorphic model dominating a given surface $S$ is a surface $S_\pi$ with a bimeromorphic map $\pi:S_\pi \rightarrow S$.

\subsection{Clusters and $H$-constants}\label{sec:clusters}

Singularities of curves on a smooth surface $S$ will be described in terms of their \emph{clusters of multiple points}, in the spirit of \cite{Cas00}, i.e., taking into account the \emph{infinitely near} multiple points ---which have to be blown up in every embedded resolution. This description will allow a convenient treatment of pullback curves and their $H$-constants. We begin by recalling the notions of infinitely near points and clusters.

\begin{definition}[{\cite[3.3]{Cas00}}]
	Given a surface $S$ and a point $p \in S$, denote $\pi_p : S_p \rightarrow S$ the blowup of $S$ at $p$. Points in the exceptional curve $E_p = \pi^{-1}(p)$ are called \emph{points in the first (infinitesimal) neighborhood of} $p$. 
Iteratively, a point $q$ in the $k$-th neighborhood of $p$ is defined as a point in the first neighborhood of a point in the $(k-1)$-th neighborhood of $p$. 
Note that in this case, the point in the $(k-1)$-th neighborhood is uniquely determined; it is called \emph{the immediate predecessor of} $q$. More generally, for every $1\le i \le k-1$, $q$ is in the $i$-th neighborhood of a unique point in the $(k-i)$-th neighborhood of $p$. The point $p$ itself can be considered to be its $0$-th neighborhood.
\end{definition}

On every blowup such as $\pi_p : S_p \rightarrow S$, it is convenient and natural to identify each point $q\in S$, $q\ne p$, with its unique preimage in $S_p$.
To do such identifications consistently across different blowups, and more generally across  bimeromorphic models dominating $S$, we shall rely on \emph{infinitely-near-ness}, a pre-order relation between points on such models. Points in the infinitesimal neighborhoods of $p\in S$ provide paradigmatic instances.
The equivalence relation induced by the pre-order will provide the desired identification of points, so the set of equivalence classes inherits a partial ordering by infinitely-near-ness.

\begin{definition}
	If $\pi_1:S_{\pi_1}\rightarrow S$, $\pi_2:S_{\pi_2}\rightarrow S$ are bimeromorphic maps, $q_1\in S_{\pi_1}$, and $q_2\in S_{\pi_2}$, then $q_2$ is \emph{infinitely near} to $q_1$  (we also write $q_2 \ge q_1$ and say $q_1$ precedes $q_2$) whenever there exist an open neighborhood $U\subset S_{\pi_2}$ of $q_2$ and a holomorphic map $\varpi:U\rightarrow S_{\pi_1}$, with $\varpi(q_2)=q_1$, such that the restriction $\pi_2|_U:U\rightarrow S$ factors as $\pi_2=\pi_1 \circ \varpi$. 
\end{definition}

In particular, every point on a bimeromorphic model, $q\in S_\pi \rightarrow S$, is infinitely near to a unique point on $S$, namely $\pi(q)$.
Obviously, if $q$ is in the $k$-th infinitesimal neighborhood of $p$ for some $k\ge 1$, then $q\ge p$.
Denote $\approx$ the equivalence relation induced by the pre-order, so that $q_1\approx q_2$ if $q_1\ge q_2$ and $q_2\ge q_1$. 

\begin{lemma}
	Let $q_1\in S_{\pi_1}$, and $q_2\in S_{\pi_2}$ be points in two bimeromorphic models of $S$. 
	Then $q_1\approx q_2$ if and only if there exist open neighborhoods $U_i\subset S_{\pi_i}$ of $q_i$ and an $S$-biholomorphism $\varpi:U_1\rightarrow U_2$, with $\varpi(q_1)=q_2$.
\end{lemma}
\begin{proof} 
	Let us prove the ``if'' part, as the ``only if'' part is obvious.
	
	By definition, $q_i\ge q_j$ if and only if there is an open neighborhood $U_i\subset S_{\pi_i}$ of $q_i$ and a holomorphic map $\varpi_i:U_i \rightarrow S_{\pi_j}$ with $p_j=\varpi_i(p_i)$ and $\pi_i|_{U_i}=\pi_j \circ \varpi_i$.
	If $p_1\ge p_2$ and $p_2\ge p_1$, then in the open neighborhood $U=U_1 \cap \varpi_1^{-1}(U_2)$ of $p_1$, the map $f=\varpi_2\circ \varpi_1:U\rightarrow U$ satisfies $\pi_1|_U=\pi_1\circ f$. 
	Since $\pi_1:S_{\pi_1}\rightarrow S$ is bimeromorphic, this implies that $f=id_U$ on a dense open subset of $U$, and hence $f$ is the identity map of $U$.
	By symmetry, $\varpi_1$ and $\varpi_2$ are mutual inverses on suitable open neighborhoods, and the claim follows.
\end{proof}

\begin{proposition}\label{pro:infinitely-near-neighbors}
	For every $p\in S$ and every $q\ge p$ there is a unique $n\ge 0$ and a unique point in the $n$-th neighborhood of $p$ equivalent to $q$.
\end{proposition}

\begin{proof}
	There is a bimeromorphic map $\pi:S_\pi\rightarrow S$ with $q\in S_\pi$ and $\pi(q)=p$. 
	First we observe that 
	\begin{enumerate}[{(}1{)}]
		\item\label{enum:1} $q$ is equivalent to a point of $S$ (which must be $q\approx p$) if and only if $\pi$ is a biholomorphism around $q$.
	\end{enumerate}
	Factor $\pi:S_\pi\rightarrow S$
	as a finite sequence of point blowups, which is possible
	\cite[III,4.4]{BHPvdV},
	%
	and denote $\{p_1,\dots,p_m\}$ the set of centers of blowups \emph{that are images of $q$}.
	We will show, by induction on $m$, that $q$ is equivalent to a point in the $m$-th neighborhood of $p$, and that $m$ and the equivalence classes of $p_1, \dots, p_m$ are independent of the factorization.
	
	The case $m=0$ follows by \eqref{enum:1}, so assume $m>0$ and $\pi$ is not a biholomorphism around $q$.
	The points $p_i$ are totally ordered by infinitely-near-ness, i.e., $p_1<\dots<p_n<q$, with strict infinitely-near-ness $<$ because of \eqref{enum:1}.

	Since $\pi(q)=p$, $p_1$ is equivalent to $p$, and since blowups of distinct proper points commute \cite[4.3.1]{Cas00}, we may rearrange the sequence of blowups so that $p_1=p$ is the center of the first blowup. This rearranging affects neither $m$ nor the equivalence class of the $p_i$.

	Now $\pi$ factors through $Bl_p(S)$, and the image of $q$ in $Bl_p(S)$ is a well defined point $r$. 
	We have $q\ge r$, and the bimeromorphic map $S_\pi \rightarrow Bl_p(S)$ factors as a finite sequence of point blowups, where the centers of blowups that are images of $q$ are $\{p_2,\dots,p_m\}$. 
	By induction it follows that $q$ is equivalent to a point in the $m$-th neighborhood of $p$.
	
	Now assume there is a second factorization, whose centers that are images of $q$ are $\tilde p_1, \dots, \tilde p_{\tilde m}$. As before, $\tilde p_1$ is equivalent to $p$ and we may in fact assume $\tilde p_1=p$. Because the blowup of $p$ is bimeromorphic, the factorization through $Bl_p(S)$ is unique, and the image of $q$ in $Bl_p(S)$ obtained from the second factorization is still $r$. By the induction hypothesis again, $\tilde m=m$ and all centers are equivalent.
\end{proof}

\begin{corollary}\label{cor:infinitely-near-neighbors}
	There is a bijection
	\[
	\frac{
		\left\{q \text{ s.t. }q\ge p \right\}}{\approx} 
	\longleftrightarrow
	\left\{\tilde q \text{ in some infinitesimal neighborhood of }p \right\}.
	\]
\end{corollary}

\begin{definition}
	In the sequel, we shall identify equivalent points; thus for us a point infinitely near to $p$ is by definition an equivalence class of points in bimeromorphic models of $S$ mapping to $p$. 
	Infinitely-near-ness is then a partial order on the set of points infinitely near to $p$. 
	\emph{An infinitely near point of} $S$ is a point infinitely near to some $p\in S$.
	If $q$ is a point in a bimeromorphic model of $S$, we will denote the infinitely near point it determines by the same symbol $q$, recalling that equality of infinitely near points means equivalence of points in models of $S$.
\end{definition}

We also observe that it follows from the proof of the previous proposition that a point in the $n$-th neighborhood of $p$ is infinitely near to exactly $n+1$ points infinitely near to $p$ (including $p$ and $q$).
Sometimes we call points $p\in S$ \emph{proper points} of $S$.

\begin{definition}
\emph{A cluster based at} $p$ is a finite set of points $K$ infinitely near to $p$ such that, for every $q\in K$, if $q'$ is a point infinitely near to $p$ and $q$ is infinitely near to $q'$, then $q' \in  K$. \emph{A multi-cluster} is a finite union of clusters based at distinct points of $S$. 
\end{definition}

By Proposition \ref{pro:infinitely-near-neighbors} and its Corollary \ref{cor:infinitely-near-neighbors}, our notion of cluster agrees with the one in \cite{Cas00}.

A curve $C$ is said to go through the infinitely near point $q\in S_\pi\rightarrow S$ if its strict transform in $S_\pi$ goes through $q$. The property is well defined, because clearly if $q'\in S_{\pi'}\rightarrow S$ is equivalent to $q'$, then the strict transform of $C$ in $S_\pi$ goes through $p$ if and only if the strict transform of $C$ in $S_{\pi'}$ goes through $p'$; in the sequel we implicitly leave such routine checks to the reader. 
For instance, the multiplicity of $C$ at $q$, denoted $\mult_q C$, is well defined as the multiplicity of its strict transform.
For every curve $C$ on $S$, the set of all points infinitely near to $p$, where $C$ has multiplicity $>1$, is a cluster \cite[3.7.1]{Cas00}, which we denote $\MSing_p(C)$, and the set of all points, proper and infinitely near, where $C$ has multiplicity $>1$, is a multi-cluster $\MSing(C)$. 
\begin{remark}
	The multi-cluster $\MSing(C)$ just defined may be strictly contained in the multi-cluster of \emph{singular points} of $C$ \cite[Section 3.8]{Cas00}, which is formed by all points that have to be blown up to obtain an \emph{embedded resolution} (also called \emph{good resolution} in the literature) of $C$.
\end{remark}

Given a multi-cluster $K$ on $S$, one may \emph{blow up all points in $K$}, as follows. 
First blow up $S$ with center at one of the proper points $p$ of $K$, then perform successive blowups on the resulting surfaces, with centers which belong to $K$ and are proper points of the surfaces obtained by previous blowups. 
Subsequent centers may be chosen in any order compatible with the natural ordering by infinitely-near-ness (if $q_1$ precedes $q_2$, then $q_1$ must be blown up first); the final surface and bimeromorpic map obtained as the composition of all blowups, which will be denoted $\pi_K:S_K\rightarrow S$, are independent on the order of these blowups -- up to unique $S$-biholomorphism (a detailed proof in the case of a single cluster can be found in \cite[Proposition 4.3.2]{Cas00}, the general case follows easily). 

In fact, every bimeromorphic model of $S$ is the blowup of all points in a convenient cluster: if $\pi\colon S_\pi\rightarrow S$ is a bimeromorphic map, for every factorization of $\pi$ as a finite sequence of point blowups, the centers of the blowups clearly form a multi-cluster $K$.	
The proof of Proposition \ref{pro:infinitely-near-neighbors} can be easily modified to show that this multi-cluster is independent of the factorization (i.e., two distinct factorizations consist of the same number of blowups, and the centers are equivalent) and there is a unique $S$-biholomorphism $S_\pi\cong S_K$.

We denote by $E_q$ (respectively, $\tilde E_q$) the pullback or total transform (respectively, the strict transform) in $S_K$ of the exceptional divisor of the blowup centered at $q$. It is not hard to see that $q_1$ precedes $q_2$ if and only if $E_{q_2}-E_{q_1}$ is an effective divisor.

\begin{definition}
	Let $C \subset S$ be a reduced curve, and let $K$ be a multi-cluster on the surface $S$. \emph{The Harbourne constant of} $C$ at $K$ is
	$$H(C,K) = \frac{C^{2} - \sum_{q\in K} \mult_{q}(C)^{2}}{|K|}.$$
	Note that the strict transform of $C$ on the blowup $S_K$ of all points in $K$ is
	\begin{equation}\label{eq:strict-transform}
	\tilde{C}=\pi_K^*(C)-\sum_{q\in K}\mult_{q}(C) E_q
	\end{equation}so the numerator in the definition of $H(C,K)$ is the self-intersection of $\tilde{C}$ in $S_K$.
	We define \emph{the Harbourne index} of $C$ as its Harbourne constant at its cluster of multiple points, i.e.,
	\[h(C)=H(C,\MSing(C)).\]
\end{definition}

If the singularities of $C$ are ordinary, then the multi-cluster $\MSing(C)$ consists of proper points of $S$ only, so this definition of Harbourne index extends the one recalled in the introduction.

\begin{remark}\label{rmk:h-4}
	Fix a reduced curve $C$ on $S$.
	For every multi-cluster $K$ and every point (proper or infinitely near) $q$ of $S$, let $K+q$ be the minimal multi-cluster which contains both $K$ and $q$. 
	Note that all points preceding $q$ which are not in $K$ belong to $(K+q)\setminus K$.
	Assume $K$ is such that there is a point $q\in \MSing(C) \setminus K$. 
	Then $C$ has multiplicity at least 2 at all points in $(K+q)\setminus K$, i.e., $\mult_q(C)^2\ge4$ for every such point. 
	Therefore, if  $H(C,K)\ge -4$, then 
	$$H(C,K+q)\le\frac{C^2-\sum_{q\in K}\mult_q(C)^2-4|(K+q)\setminus K|}{|K|+|(K+q)\setminus K|}\le H(C,K),$$ 
	and by induction on $|\MSing(C)|-|K|$,  if  $H(C,K)\ge -4$, then $H(C,K)\ge H(C,\MSing(C))$.
	On the other hand, in the case of $S = \mathbb{P}^{2}$  every known value $H(C,K)$ is larger than $-4$,	so in all known cases for plane curves the cluster that gives the smallest value for $H(C,K)$ is $K=\MSing(C)$, and the value is $h(C)$.
\end{remark}


\subsection{Singularities of curves in smooth surfaces}\label{sec:singularities}

%
%
%

By assigning integral multiplicities $\nu=\{\nu_q\}_{q\in K}$ to the points of a cluster $K$ one gets a \emph{weighted cluster}.
\begin{definition}
A weighted cluster $\mathcal{K}=(K,\nu)$ is \emph{consistent} if there exist germs of curve in $S$ whose strict transform at each $q\in K$ has multiplicity exactly $\nu_q$. The cluster $\MSing_p(C)$ of multiple points on $C$ infinitely near to $p\in S$, weighted with the multiplicity of $C$ at each point, will be denoted by $\WSing_p C$.
\end{definition}
Let $\mathcal{K}=(K,\nu)$ be a given weighted cluster of points infinitely near to $p$, and $\pi_K:S_K\rightarrow S$ the blowup of $S$ at all points of $K$, introduced above. 
Continuing to denote by $E_q$ the total transform in $S_K$ of the exceptional divisor above each $q\in K$, we associate to the weights $\nu$ an effective divisor on $S_K$, $D_{\mathcal{K}}=\sum \nu_q E_q$. 
\begin{definition}
A curve $C$ on $S$ is said to go through the weighted cluster $\mathcal{K}=(K,\nu)$ if $\pi_K^*(C)-D_{\mathcal{K}}$ is effective in $S_K$ (see \cite[Chapter 4]{Cas00}).
\end{definition}
In particular, if the strict transform of a curve $C$ at every $q\in K$ has multiplicity equal to $\nu_q$, then $C$ goes through $\mathcal{K}$. 
The complete ideal $\mathcal{H}_{\mathcal{K}}={\pi_K}_*(-D_{K,\nu})\subset \cO_{S,q}$ is formed by the equations of germs of curve at $p$ going through $\mathcal{K}$. 
\begin{definition}
The cluster $(K,\nu)$ is consistent if and only if it satisfies Enriques' \emph{proximity inequalities}, which can be stated as the non-positivity of the intersection numbers $D_{\mathcal{K}}\cdot \tilde E_q\le 0$ on $S_K$ for every $q \in K$, see  \cite[sections 4.2 and 4.5]{Cas00}. 
\end{definition}
The self-intersection of $\mathcal{K}$ is defined as the opposite of the self-intersection of $D_{\mathcal{K}}$: $\mathcal{K}^2=\sum \nu_q^2$. 
If $\mathcal{K}$ is consistent, then its self-intersection equals the intersection multiplicity at $q$ of two sufficiently general germs of curve going through $\mathcal{K}$ with multiplicities equal to $\nu$ \cite[3.3.1 and 4.2.3]{Cas00}.

The notions of weighted cluster, and hence of going through a weighted cluster, consistency and self-intersection carry over to the multi-cluster setting verbatim.
We will denote $\WSing (C)$ the weighted multi-cluster obtained as the union of all weighted clusters $\WSing_p C$, where $p\in S$ is a singular point of $C$. 

\begin{definition}
For an $\mathfrak{m}_p$-primary ideal $I\subset \cO_{S,p}$, the Hilbert-Samuel multiplicity of $I$ is the number $e(I)$ such that, for $k \gg0$,
\[ \dim \frac{\cO_{S,p}}{I^k} = e(I) \frac{k^2}{2!}+O(k). \]
\end{definition}
\begin{lemma}
	The Hilbert-Samuel multiplicity of a complete ideal $\mathcal{H}_{\mathcal{K}}$, where $\mathcal{K} = (K,\nu)$ is a consistent cluster, is $e(\mathcal{H}_{\mathcal{K}})=\mathcal{K}^2$.
\end{lemma}
\begin{proof}
	For every positive integer $k$, let $k\mathcal{K}=(K,k\nu)$ be the weighted cluster consisting of the same points as $\mathcal{K}$ and all weights multiplied by $k$. 
	It satisfies the proximity inequalities, so it is consistent.
	By \cite[Theorem 8.4.11]{Cas00}, $\mathcal{H}_{\mathcal{K}}^k=\mathcal{H}_{k\mathcal{K}}$. 
	On the other hand, by the codimension formula for consistent clusters
	\cite[Proposition 4.7.1]{Cas00}, one has
	\[\dim \frac{\cO_{S,p}}{\mathcal{H}_{K,k\nu}}=
	\sum_{p\in K} \frac{k\nu_p(k\nu_p+1)}{2}
	=\mathcal{K}^2\frac{k^2}{2} + O(k).\qedhere\]
\end{proof}

\begin{lemma}\label{lem:H-passing}
	Let $S$ be a smooth projective surface, 
	let $p_1, \dots, p_r \in S$, and for every $i=1,\dots, r$ we denote by $\mathcal{K}_i=(K_i,\nu^{(i)})$ a consistent weighted cluster of points infinitely near to $p_i$. Let $\mathcal{K}=(K,\nu)$ be the multi-cluster formed by all these clusters and $C\subset S$ a reduced curve going through $\mathcal{K}$. Then
	\[H(C,{K}) \le \frac{C^2 -\mathcal{K} ^2}{|K|}\]
\end{lemma}
\begin{proof}
	Denote by $\mu_q^{(i)}$ the multiplicity of the strict transform of $C$ at the point $q\in K_i$. Then the clusters $\mathcal{K}_i'=(K_{i},\mu^{(i)})$, for $i=1\dots,r$, are consistent. Moreover,
	the strict transform of $C$ on the blowup $\pi_K:S_K\rightarrow S$ at all points of the multi-cluster $K$ is 
	\(\tilde C=\pi^*(C)-D_{\mathcal{K}'}\)
	whereas, since $C$ goes through all clusters with multiplicities $\nu^{(i)}$, 
	the divisor
	\(\pi^*(C)-D_{\mathcal{K}}\) 
	is effective. It follows that for each $i$, $D_{\mathcal{K}_i'}\ge D_{\mathcal{K}_i}$, and therefore there is an inclusion of ideals 
	$$\mathcal{H}_{\mathcal{K}_i'}=\pi_{K_i *}(-D_{\mathcal{K}_i'})\subset \pi_{K_i *}(-D_{\mathcal{K}_i})=\mathcal{H}_{\mathcal{K}_i}$$
		in $\cO_{S,p_i}$ for $i=1,\dots, r$. Thus for every $k$, $\mathcal{H}_{\mathcal{K}_i'}^k\subset \mathcal{H}_{\mathcal{K}_i}^k$, so
	the Hilbert-Samuel multiplicities satisfy $e(\mathcal{H}_{\mathcal{K}_i'})\ge e(\mathcal{H}_{\mathcal{K}_i})$, and by the lemma above $(\mathcal{K}_i')^2\ge \mathcal{K}_i^2$. Finally, this implies
	\[H(C,{K}) = \frac{C^2 -\mathcal{K'} ^2}{|K|}\le \frac{C^2 -\mathcal{K} ^2}{|K|}\]
	as wanted.
\end{proof}

\section{Harbourne constants under ramified morphisms}

Our next goal is to describe the singularities of preimages of curves under ramified holomorphic surface maps, in enough detail to first show that $H$-constants can only drop under such a process on projective surfaces, and secondly to provide new examples of plane curve arrangements in the complex projective plane with very negative $H$-indices. 

\subsection{The pullback cluster}\label{sec:pullback_cluster}

Fix for this section the following notation: $S, S'$ are two smooth complex surfaces, $f:S\rightarrow S'$ is a dominant holomorphic map (i.e., $f(S)$ is not contained in a curve of $S'$), $p\in S$ and $p'=f(p)$ are points, and we are interested in the singularity at $p$ of the pullback (or preimage) of a curve $C'\subset S'$ whose singularity at $p'$ is known.

Take $x, y$ and $u, v$ as local coordinates on $S$ and $S'$ with origins at $p$
and $p'$, respectively, and assume that on a suitable open neighborhood $U$ at $p$, $f:S\rightarrow S'$ is given by the equalities $u = f_1 (x, y)$, $v = f_2 (x, y)$,
where $f_i$'s are non-invertible convergent series in $x, y$. 
The \emph{multiplicity} of $f$ at $p$, denoted by $\nu_p(f)$ or simply $\nu(f)$ if there is no risk of ambiguity, is the minimum of the orders of vanishing of $f_1$ and $f_2$ at $p$.

Put $d=\gcd(f_1,f_2)$, as elements in $\mathcal{O}_{S,p}$. The pencil of curves in $U$ defined by $\{C_\alpha:\alpha_1 f_1+\alpha_2 f_2=0\}$, $\alpha=\alpha_1/\alpha_2\in \bbC\cup\{\infty\}$, formed by the pullbacks of the curves $\alpha_1 u + \alpha_2 v=0$, has a fixed part $F:d=0$ (which might be empty) and a variable part $\{D_\alpha:\alpha_1 f_1/d+\alpha_2 f_2/d=0\}$. 
We call $F$ the \emph{curve contracted to $p'$}, as $f(F)=p'$.
On the other hand, the variable part of the pencil has, like every pencil without fixed part, a weighted cluster of base points which consists of the points and multiplicities shared by all but finitely many curves in the pencil \cite[7.2]{Cas00}.
This cluster is called the \emph{cluster of base points} of $f$, and denoted $\BP_p(f)$, or simply $\BP(f)$ if no confusion is likely.
By definition, $\BP(f)$ is a consistent cluster, and for every curve $C'\subset S'$ through $p'$, the pullback $f^*(C')$ goes through $\BP(f)$ (if $F$ is nonempty, then $f^*(C')-F$ goes through $\BP(f)$).
It may happen that only finitely many $D_\alpha$ go through $p$; in this case the cluster $\BP(f)$ is empty. 
The multiplicity of $f$ satisfies $\nu(f)=\nu_p+\mult_{p}(F)$, where $\nu_p$ is the multiplicity of $p$ in $\BP(f)$.


\begin{remark}\label{rmk:BP-discrete}
	Given a dominant holomorphic map $f:S\rightarrow S'$ and a point $p'\in S'$. The set of points $p\in S$ such that $f(p)=p'$ and $\BP_p(f)$ is nonempty is discrete.
	Indeed, let $p$ satisfy $f(p)=p'$, and write $f$ in local coordinates as above, $(f_1(x,y),f_2(x,y))$ in a neighborhood $U$ of $p$. Let $d=\gcd(f_1,f_2)$.
	Since $f_1/d$, $f_2/d$ have no common factor in $\mathcal{O}_{S,p}$, their common zeros in a possibly smaller neighborhood $U'\subset U$ are a discrete set, and for $q\in U'$, the cluster $\BP_{q}$ is nonempty if and only if $q$ is a common zero of $f_1/d$ and $f_2/d$.

	Note that if $d$ is invertible in $\mathcal{O}_{S,p}$, then $p$ is an isolated preimage. 
\end{remark}

Let $\pi_{p'}:{S}_{p'}'\rightarrow S'$ be the blowup centered at $p'$. 
It is natural to describe $\BP(f)$ as the cluster of points which need to be blown up to resolve the indetermination at $p$ of the ``meromorphic map'' $\tilde{f}=\pi_{p'}^{-1}\circ f:S \dashrightarrow S_{p'}$; we include a proof for completeness, since this characterization will be the starting point for our definition of the pullback cluster.


\begin{lemma}\label{lem:BP-as-resolution}
	Keeping the same notation as above, assume $f(p)=p'$ and let $U$ be a neighborhood of $p$ such that $\BP_{q}(f)$ is empty for all $q\in U$, $q\ne p$.
	There is a unique local lift $\tilde{f}$ of $f$ to the blowup of the points of $\BP(f)$ which makes the following diagram commute:
	\[\begin{tikzcd}
	U_{\BP(f)} \rar{\tilde{f}} \dar[swap]{\pi_{\BP(f)}}  & {S}'_{p'} \dar{\pi_{p'}} \\
	U \rar[swap]{f} & S'  \\
	\end{tikzcd}
	\]
	Moreover, if $\pi:U_\pi\rightarrow U$ is a bimeromorphic model of $U$ which admits a lift $\bar f:U_\pi\rightarrow S'_{p'}$, then $\pi$ factors through $\pi_{\BP(f)}$.
	
	The weights $\nu_q$ of the base points are determined by the formula
	\begin{equation}\label{eq:base_cluster_multiplicities}
	\tilde{f}^*(E_{p'})=\pi_{\BP(f)}^*(F)+ \sum_{q\in \BP(f)} \nu_q E_q.
	\end{equation}
\end{lemma}
\begin{remark}\label{BP-after-blowup}
	If $q$ is a point infinitely near to $p$, which belongs as a proper point to the model $S_\pi\rightarrow S$, then $(f\circ \pi)(q)=p'$. 
	It then follows from the definition that $q\in\BP_p(f)$ if and only if $q\in \BP_q(f\circ \pi)$ (see also \cite[section 7.2]{Cas00}).
\end{remark}
\begin{proof}
	Write $f$ in local coordinates as above, $(u,v)=(f_1(x,y),f_2(x,y))$, and let $d=\gcd(f_1,f_2)$.
	
	Assume first that $\BP(f)$ is empty, which by definition means that either $f_1/d$ or $f_2/d$ does not vanish at $p$; without loss of generality we assume $f_2/d$ does not vanish.
	Consider the chart $V$ of $S'_{p'}$ which admits $(u/v,v)$ as local coordinates, and let $U'\subset U$ be the open set where $f$ is given by $(f_1, f_2)$ and $f_2/d$ does not vanish.
	Then the restriction $f|_{U'}$ lifts uniquely to $\tilde{f}:U'\rightarrow V$, given in coordinates as $(u/v,v)=(f_1(x,y)/f_2(x,y),f_2(x,y))$.
	
	Conversely, if there is a neighborhood $U'$ of $p$ where $f$ lifts to $\tilde f:U'\rightarrow S'_{p'}$, then $\tilde{f}(p)$ belongs either to the chart where $(u/v,v)$ are local coordinates, or to the chart where $(u,v/u)$ are local coordinates; without loss of generality we assume it is the first case.
	Then $(u/v)\circ \tilde f=f_1(x,y)/f_2(x,y)$ is a regular function in a neighborhood of $p$, so $f_2/d$ does not vanish and $\BP(f)$ is empty.
	
	So $\BP(f)$ is empty if and only if there is a lift $\tilde{f}:U'\rightarrow S'_{p'}$ in some neighborhood $U'$ of $p$. 
	We now observe that uniqueness of a lift $\tilde f$ if it exists is clear because $\pi_{\BP(f)}$ is bimeromorphic. 
	Therefore, to show existence in the general case it is enough to prove it in a neighborhood of a point $q\in U_{\BP(f)}$, because by uniqueness local lifts will match to the desired $\tilde f$.
	
	Now, by assumption $\BP_{q}(f)$ is empty for all $q\in U, q\ne p$, so $\BP_{q}(f\circ \pi_{\BP(f)})$ is empty for all $q\in U, f(q)\ne p$, and by Remark \ref{BP-after-blowup}, $\BP_{q}(f\circ \pi_{\BP(f)})$ is empty for all $q\in U, f(q)= p$.
	Therefore the previous paragraph shows that there is a lift $\tilde f$ as claimed.
	
	Moreover, if $p\in\BP(f)$ then as observed above there is no lift of $f$ to any neighborhood $U'$ of $p$. 
	Therefore, if $\pi:U_\pi\rightarrow U$ is a bimeromorphic model of $U$ which admits a lift $\bar f:U_\pi\rightarrow S'_{p'}$, then $\pi$ is not biholomorphic onto any neighborhood of $p$, and so it factors through $\pi_p$. The second claim now follows by induction on $|\BP(f)|$.
	
	Let now $(\bar\nu_q)_{q\in \BP(f)}$ be the weights determined by \eqref{eq:base_cluster_multiplicities}. 
	It remains to show that $\nu_q=\bar{\nu}_q$ for each $q$, which we do by induction on $|\BP(f)|$.
	To this end, let $q$ be a point in the first neighborhood of $p$, not in $\BP(f)$ nor on $F$; this means that not all (strict transforms of) $D_\alpha$ go through $q$, and without loss of generality we assume that the strict transform of $f_2/d$ does not go through $q$. Then a direct computation in coordinates shows that the pullback $\pi_p^*(f_2)$ of $E_{p'}:v=0$ vanishes to order exactly $\nu_p$ along $E_p$ at $q$, i.e.,  $\nu_p=\bar{\nu}_p$.
	This in particular gives the desired equality if $\BP(f)$ consists only of the point $p$. 
	Finally, let $q_1,\dots,q_r$ be the points of $\BP(f)$ in the first neighborhood of $p$; the induction hypothesis applied to $\BP_{q_i}(f\circ \pi_p)$ finishes the proof. \end{proof}

According to \cite{Cas07}, the \emph{local degree} of $f$ at $p$, denoted by $\deg_p(f)$, is the number of points in $f^{-1}(q)$ that approach $p$ when $q$ approaches $p'$ along $\alpha_1 u + \alpha_2 v=0$ for a general $\alpha$. If the contracted curve $F$ is empty, then this is simply the number of points in $f^{-1}(q)$ that approach $p$ when $q$ approaches $p'$. In general, it satisfies
\begin{equation}\label{eq:local_degree}
\deg_p(f)=\sum_{q\in \BP_p(f)} \left(\nu_q^2 +\nu_q\cdot\mult_q(F)\right),
\end{equation}
where $\nu_q$ are the weights of the cluster of base points $\BP_p(f)$. 
Note that $\sum \nu_q^2=\BP_p(f)^2$ is the intersection multiplicity at $p$ of any two distinct curves $D_\alpha, D_{\alpha'}$ in the pencil of variable parts \cite[Ex. 7.2]{Cas00}, whereas $\sum \nu_q \cdot \mult_q(F)$ is the intersection multiplicity at $p$ of a general $D_\alpha$ with the fixed part $F$.
\begin{remark}\label{rmk:degree_composition}
By definition, the local degree is multiplicative under composition of finite maps. More precisely, if $f:S\rightarrow S'$ and $g:S'\rightarrow S''$ are holomorphic maps and $f$ has empty contracted curve, denoting $p'=f(p)$, one has $\deg_p(g\circ f)=\deg_{p'}(g)\cdot\deg_p(f)$.
\end{remark}

Let $\mathcal{K}=(K,\mu)$ be an arbitrary weighted cluster of points infinitely near to the target point $p'$ of $f$. 
Generalizing the cluster of base points of $f$ which describes the singularities of pullbacks of general curves smooth at $p'$, we next associate to $\mathcal{K}$ a \emph{pullback cluster} $f^*(\mathcal{K})=(f^*(K),f^*\mu)$ of points infinitely near to $p$ in order to describe the singularities of pullbacks of general curves going through $\mathcal{K}$.
Following the characterization of $\BP(f)$ given as Lemma \ref{lem:BP-as-resolution}, let $\pi_{K}:S'_{K}\rightarrow S'$ be the composition of the blowups centered at all points of $K$, and let $D_{\mathcal{K}}=\sum_{q\in K} \mu_q E_q$ be the associated divisor on $S'_{K}$. 
Then we define $f^*(K)$ as the cluster of all points which need to be blown up to resolve the indeterminacy at $p$ of the composition $\pi_{K'}^{-1}\circ f$, and its multiplicities $f^*\mu$ as determined by 
\begin{equation*}
\tilde{f}^*(D_{\mathcal{K}})=F_{\mathcal{K}}+ \sum_{q\in f^*(K)} (f^*\mu)_q E_q.
\end{equation*}
where $F_{\mathcal{K}}$ is formed by all components of $\tilde{f}^*(D_{\mathcal{K}})$ which do not contract to $p$ in $S$.
$F_{\mathcal{K}}$ can be called the fixed part of $f^*(\mathcal{H}_{\mathcal{K}})$; all of its components map in $S$ to components of the contracted curve $F$ of $f$, with multiplicities depending on $\mathcal{K}$.
If the weighted cluster $\mathcal{K}$ is consistent, the pullback cluster $f^*(\mathcal{K})$ can be described as a suitable cluster of base points; the following lemma gives the precise statement. 

\begin{lemma}\label{lem:pullback-as-BP}
Let $w,z\in \cO_{S',p'}$ be local equations of two curves going through $\mathcal{K}$ with multiplicities exactly $\mu$ and sharing no further point (see \cite[Section 4.2]{Cas00}).
Let $V\subset S'$ be an open neighborhood of $p'$ such that $(w, z)$ determine a dominant holomorphic map $g:V\rightarrow \bbC^2$.
Then $F_\mathcal{K}$ is the pullback in $S_{f^*(K)}$ of the contracted germ of the holomorphic map $g\circ f$, and $f^*(\mathcal{K})$ differs from the cluster of base points at $p$ of $g\circ f$ at most in some points of multiplicity 0.
\end{lemma}
\begin{proof}[Sketch of proof.]
The key point of the proof is to show that there is a lift $U_{f^*(K)}\rightarrow Bl_0(\mathbb{C}^2)$ of $g \circ f$, which follows from the straightforward fact that $g$ lifts to $Bl_K(V)\rightarrow Bl_0(\mathbb{C}^2)$  (in fact by the definitions $K$ is the cluster of base points of $g$).
\end{proof}

Put $d=\gcd(f^*(w),f^*(z))$. It follows from the lemma that the curve $F_{\mathcal{K}}$ is given by $d=0$, and the multiplicities of all but finitely many curves in the pencil $\{\alpha_1 f^*(w)/d+\alpha_2f^*(z)/d\}$ at the points of $f^*(K)$ are exactly the weights $f^*\mu$. 
More precisely, the cluster of base points of this pencil consists of the subcluster of $f^*\mathcal{K}$ of the points with positive multiplicity. If $\mathcal{K}$ is consistent, then $f^*\mathcal{K}$ is consistent as well. However, when $\mathcal{K}$ has points $q$ whose excess $D_{\mathcal{K}}\cdot \tilde{E}_q$ is zero  \cite[Section 4.2]{Cas00}, $f^*\mathcal{K}$ may have points with multiplicity zero.

\begin{corollary}\label{cor:multiplicative}
	Let $f:S\rightarrow S'$ be a dominant holomorphic map between smooth complex surfaces, $p\in S$ a point with $f(p)=p'$, and let $\mathcal{K}=(K,\mu)$ be a consistent weighted cluster of points infinitely near to $p'$. 
	If the curve contracted to $p'$ is empty, then 
	\[ (f^*\mathcal{K})^2 = \deg_p (f) \cdot {\mathcal{K}}^2. \]
\end{corollary}
\begin{proof}
	Let $w,z\in \cO_{S',p'}$ be the local equations of two distinct curves going through $K$ with multiplicities exactly $\mu$, and sharing no further point. 
	Consider the dominant holomorphic map $g:V\rightarrow \bbC^2$ determined by $(w, z)$ as above.
	Since it has empty contracted curve, by \eqref{eq:local_degree}, $\deg_{p'}(g)=\mathcal{K}^2$, and by the lemma, $\deg_p(g\circ f)=(f^*\mathcal{K})^2$. Since by Remark \ref{rmk:degree_composition}, $\deg_p(g\circ f)=\deg_p (f) \cdot\deg_{p'}(g)$,
	the claim follows.
\end{proof}
\begin{proposition}\label{pro:submultiplicative}
	Let $f: S\rightarrow S'$ be a dominant holomorphic map between smooth complex surfaces, and $p\in S$ such that $f^{-1}(p')=\{p\}$. 
	For every cluster $K$ of points infinitely near to $p'$,
    $| f^*(K) |\le \deg_p (f)\cdot|K|$.
	Moreover, if the multiplicity of $f$ is $\nu(f)>1$, then the inequality is strict.
\end{proposition}
\begin{proof}
	Note that $p$ is an isolated preimage of $p'$, so in particular the curve contracted to $p'$ is empty.
	
	We argue by induction on $|K|$. If $|K|=1$, then $K=\{p'\}$ and $f^*(K)=\BP(f)$. 
	Since $\deg_p(f)=\sum_{q\in \BP(f)} \nu_q^2$ and $\nu_q\ge 1$ for all $q \in \BP(f)$ by definition, the claims follow. 
	
	Now assume $|K|>1$, and let $q_0'$ be a maximal point by the partial ordering by infinitely-near-ness, so that $K_0 = K\setminus\{q_0'\}$ is a cluster. 
	By induction we may assume that 
	\begin{equation}\label{eq:induction}
	| f^*(K_0)|\le \deg_p (f)\cdot|K_0|=\deg_p (f)\cdot
	\left(|K|-1\right).
	\end{equation}
	
	Since 
	$p$ is an isolated preimage of $p'$, there exist open neighborhoods $U\subset S$ of $p$ and $V\subset S'$ of $p'$ such that $f|_{U}: U\rightarrow V$ is a surjective proper holomorphic map. 
		Then 
		\begin{equation}\label{eq:degree_mumford}
		\deg_p(f)=|f^{-1}(q')|, \forall q' \in V \setminus \{p'\}
		\end{equation}
		(see \cite[\S 3.A]{Mum76}). 
		Consider the blowups $\pi_{f^*(K_0)}$ and $\pi_{K_0}$ of all points in the clusters $f^*(K_0)$, $K_0$, and the corresponding lift of $f$,
\[\begin{tikzcd}
\tilde{U} \rar{\tilde{f}} \dar[swap]{\pi_{f^*(K_0)}}  & \tilde V \dar{\pi_{K_0}} \\
U \rar[swap]{f} & V  \\
\end{tikzcd}
\]
%
		The point $q_0'$ belongs to $\tilde V$. More precisely, to the preimage of $p'$, which is an effective divisor $D$ in $\tilde V$, and
		$\tilde{f}^{-1}(q_0')$ is contained in the divisor $\tilde{f}^*(D)$, which is the preimage of $p$ in $\tilde U$.
		
		Choose local coordinates $u,v$ in an open neighborhood 
		$V'\subset \tilde V$ of $q_0'$, such that $uv=0$ along $D\cap V'$
		(this is possible because at most two prime components of $D$
		meet at $q_0'$). Then for a general member of the pencil 
		$\{L_\alpha: \alpha_1 u+\alpha_2 v=0\}$, every point on $L_\alpha$ except $q_0'$
		has exactly $\deg_p (f)$ preimages by $\tilde{f}$, 
		\eqref{eq:degree_mumford}. 
		The set $\tilde{f}^{-1}(q_0')$ need not be finite, 
		but it is easy to see that 
		\[\mathcal{Q}=\{q\in \tilde U \,|\, \tilde{f}(q)=q_0', 
		\deg_q(\tilde{f})>0\}\]
		is finite, and in fact
		\[ \sum_{q\in \mathcal{Q}} \deg_q(\tilde{f})\le \deg_p(f). \]
        Notice that $\mathcal{Q}$ is also the set of indeterminacy points of $\pi_{q_0'}^{-1}\circ \tilde f$, where $\pi_{q_0'}$ is the blowup
		centered at $q_0'$.
		Consider, for each $q \in \mathcal{Q}$, the cluster $\BP_q(\tilde{f})$. 
		It is clear that blowing up $\tilde U$ at all points of
		$\bigcup_{q\in \mathcal{Q}} \BP_q(\tilde{f})$ resolves the indeterminacies
		of $\pi_{q_0'}^{-1}\circ \tilde f$. 
		Therefore, blowing up all points in the cluster
		\[K_1=f^*(K)_0 \cup \bigcup_{q \in \mathcal{Q}} \BP_q(\tilde{f})\] 
		resolves the indeterminacies
		of $\pi_{K}^{-1}\circ  f=\pi_{q_0'}^{-1}\circ\pi_{K_0}^{-1} \circ f$. 
		By definition, it follows that $f^*(K) \subset K_1$, and therefore
		\[
		|f^*(K)| \le|K_1| \le |f^*(K)_0|+ \sum_{q \in \mathcal{Q}} \deg_q \tilde f \le \deg_p (f)\cdot
		\left(|K|-1\right) +\deg_p (f)
		\]
		so we are done. 
		Note that if $\nu(f)>1$, then (by the induction hypothesis) the inequality
		$|f^*(K)_0|<\deg_p (f)\cdot\left(|K|-1\right)$ is strict and hence also
		$|f^*(K)|<\deg_p (f)\cdot |K|$.		
\end{proof}

\subsection{Pullback of multi-clusters and $H$-indices}

We are now going to apply the local results from the previous section in the global setting of finite morphisms $f:\bbP^2\rightarrow\bbP^2$ in order to study the behavior of $H$-constants under pullbacks.

\begin{theorem}[Theorem {\ref{intro:pullback-H}}]\label{thm:pullback-H}
Let $f:S \rightarrow S'$ be a finite morphism of complex projective surfaces, $C\subset S'$ a reduced curve, and $K$ a multi-cluster on $S'$. Assume that $f^*(C)$ is reduced and $H(C,K)\le 0$. Then
\[H(f^*(C),f^*(K))\le H(C,K),\]
with a strict inequality if there is a proper point $p \in f^{*}(\mathcal{K})$ with $\nu_p (f)>1$.
\end{theorem}
\begin{proof}
	Let $k=\deg(f)$. By definition,
	\begin{align*}
	H(C,K)&=\frac{C^2-\sum_{q\in K}(\mult_q(C))^2}{|K|}, \\
	H(f^*(C),f^*(K))&=\frac{k\, C^2-\sum_{q\in f^*(K)}(\mult_q(f^*(C)))^2}{|f^*(K)|},
	\end{align*}

	
	Let $K_i$, for $i=1, \dots, r$, be the clusters composing $K$, with $K_i$ a cluster of points infinitely near to $p_i\in S'$. 
	For each $p\in S$ with $f(p)=p_i$, denote $f_p$ be the restriction of $f$ to a neighborhood $U_p$ of $p$ such that $f^{-1}(p_i)\cap U_p=\{p\}$.
	Consider the cluster $(K_i,\MSing(C))$ weighted with the multiplicities of $C$ at the points of $K_i$, and let $\mathcal{K}_p$ be the weighted cluster $f^*(K,\MSing (C))$ obtained by the pullback.
	The properties of the pullback cluster from the previous section immediately give that:
	\begin{enumerate}
		\item $\mathcal{K}_p$ is a weighted cluster of points infinitely near to $p$,
		\item $f^*(C)$ goes through $\mathcal{K}_p$,		
		\item $f^*(K)=\bigcup_{f(p)\in \{p_1,\dots,p_r\}} K_p$.
	\end{enumerate}
	Therefore,
	\begin{equation}\label{eq:pullback-H}
	H(f^*(C),f^*(K))=\frac{k\, C^2-\sum_{i=1}^r \sum_{f(p)=p_i} \mathcal{K}_p^2}{\sum_{i=1}^r \sum_{f(p)=p_i}|\mathcal{K}_p|}.
	\end{equation}
	For each $i=1,\dots, r$, Corollary \ref{cor:multiplicative} gives
	\[
	\sum_{f(p)=p_i} \mathcal{K}_p^2=\sum_{f(p)=p_i} \deg_p(f)\,(K_i,\MSing(C))^2=k\,(K_i,\MSing(C))^2,
	\]
	and Proposition \ref{pro:submultiplicative} gives
	\[
	\sum_{f(p)=p_i}|\mathcal{K}_p|\le \sum_{f(p)=p_i} \deg_p(f)\,|K_i|=k\,|K_i|,
	\]
	with a strict inequality if $\nu_{p}f>1$ at some $p$ with $f(p)=p_i$. 
	Taking into account that $H(C,K)\le 0$, the equation \eqref{eq:pullback-H} now gives
	\[
	H(f^*(C),f^*(K)) \le
	\frac{k\left(C^2-\sum_{q\in K}(\mult_q(C))^2\right)}
	{k\,\sum_{i=1}^r |K_i|},
	\]
	with a strict inequality if there exist some $i$ and $p$ with $f(p)=p_i$, and $\nu_{p}f>1$. 	
\end{proof}

For plane curves and in terms of $h$-indices, we have the following corollary:

\begin{corollary}
	Let $f:\bbP^2 \rightarrow \bbP^2$ be a finite morphism, and let $C\subset \bbP^2$ be a reduced curve with $f^*(C)$ reduced.
	\begin{enumerate}
		\item if $-4\le h(C)<0$, then $h(f^*(C))\le h(C)$,		
		\item if there is a proper point $p \in f^{-1}(\MSing(C))$, then the above inequality is strict.
	\end{enumerate}
\end{corollary}
\begin{proof}
	By the theorem, $H(f^*(C),f^*(\MSing(C))\le H(C,\MSing(C))=h(C)$, with strict inequality if there is a proper point $p \in f^{-1}(\MSing(C))$ with $\nu_p (f)>1$. By Remark \ref{rmk:h-4}, either $h(f^*(C))<-4 \le h(C)$ or $h(f^*(C))\le H(f^*(C),f^*(\MSing(C))$, and we are done. 
	%
\end{proof}

Theorem \ref{thm:pullback-H} means that $H$-constants of negative curves can only decrease under pullbacks. 
By using suitable ramified morphisms we can now prove Theorem \ref{intro:strict}.

\begin{proof}[Proof of Theorem \ref{intro:strict}]\label{proof:strict}
	Let $S_\pi \rightarrow \bbP^2$ be the composition of $n$ point blowups, and let $K$ be the multi-cluster formed by the $n$ points blown up. 
	Let $C$ be a reduced curve on $S_\pi$. 
	We want to show that $C^2/|K|>h$.
	By Lemma \ref{lem:H-passing}, we may assume that $C$ is the strict transform of a reduced curve $C'$ on $\bbP^2$.
	By Theorem \ref{thm:pullback-H}, it will be enough to show that there exists $f:\bbP^2\rightarrow \bbP^2$ satisfying:
	\begin{enumerate}
		\item $f^*(C)$ is reduced;
		\item There exists $p\in S$ such that $\nu_p(f)>1$ and $f(p)$ is a proper point of $K$.
	\end{enumerate}
	This is obviously possible: let $f:\bbP^2\rightarrow \bbP^2$ be a Kummer cover given in suitable coordinates by $f([x:y:z])=[x^n:y^n:z^n])$ with $n\ge 2$, where the coordinates are chosen such that no coordinate line is a component of $C$ (hence $f^*(C)$ is reduced) and at least one coordinate point belongs to $K$.
\end{proof}

Furthermore, we prove that there is no minimal $h$-index in any sense:

\begin{proposition}\label{pro:infimum}
	\begin{enumerate}
		\item There is no curve  $C_0\subset \bbP^2$  such that
\[h(C_0)=\inf_{C\subset\bbP^2} h(C).\]
		\item There is no curve  $C_0\subset \bbP^2$ with ordinary singularities such that
\[h(C_0)=\inf_{\substack{C\subset\bbP^2\\ \text{ordinary singularities}}} h(C).\]
	\end{enumerate}
\end{proposition}
\begin{proof}
	We shall show that, given a particular reduced curve $C\subset \bbP^2$ with $r\geq 2$ singular points and $h(C)<0$, there exists another curve $C'$ with $h(C')<h(C)$, and if $C$ has ordinary singularities then $C'$ can be chosen with ordinary singularities too.
	So let $p$ be a nonsingular point of $C$, and choose coordinates in $\bbP^2$ such that
	\begin{enumerate}
		\item $p=[1:0:0]$ 
		\item The points $[0:1:0]$ and $[0:0:1]$ do not belong to the curve, and none of the coordinate lines is tangent to $C$ or passes through a singular point of $C$ at $P$.
	\end{enumerate}
	Choose an integer $k$ such that $-k^2<h(C)$ and
	consider the Kummer cover $f:\bbP^2 \rightarrow \bbP^2$ given coordinate-wise by $f([x:y:z])=[x^k:y^k:z^k]$, which is a morphism of degree $k^2$ branched along the coordinate triangle and has multiplicity $k$ at the (fixed) point $p$. 
	The singularities of $f^*(C)$ are as follows:
	\begin{itemize}
		\item For each singular point $q$ of $C$ there are $k^2$ locally isomorphic singularities in the $k^2$ distinct preimage points of $q$.
		\item There is an ordinary singularity of multiplicity $k$ at $p$.
	\end{itemize}
	Note that, if $C$ has ordinary singularities, then so does $f^*(C)$.
	
	Denote $\mathcal{K}=\WSing(C)$ the weighted multi-cluster of multiple points of $C$. We have
	\begin{equation}\label{eq:h-pullback}
	\begin{gathered}
	h(f^*(C))=H(f^*(C),f^*(K) \cup \{p\})=\frac{k^2 d^2-k^2 \mathcal{K}^2-k^2}{k^2|K|+1}=\\
	\frac{k^2|K|\,H(C,K)-k^2}{k^2|K|+1}<\frac{k^2|K|\,h(C)+h(C)}{k^2|K|+1}
	=h(C),
	\end{gathered}
	\end{equation}
	as claimed. 
\end{proof}

 Note that Theorem \ref{intro:strict} and Proposition \ref{pro:infimum} do not mean that the values of $H$-constants are \emph{not bounded from below} -- some examples of sequences of reduced curves with decreasing $H$-indices are known, which converge to finite limits. 
 So the question whether the BNC holds for blowups of the complex projective plane is open.
 The method of Proposition \ref{pro:infimum} does mean that, if there is a uniform bound $h(C)\ge h$, then for curves with a fixed number of singular points $s$ a stronger bound than $h$ can be given:
 
 \begin{proposition}\label{pro:h-bound}
 	Suppose that $\inf_{C\subset\bbP^2}h(C)=h$ for some $h\in \bbR$.
 	Then, for every reduced curve $C\subset\bbP^2$, 
 	\[h(C)\ge h + \frac{3}{|\MSing(C)|}.\]
 \end{proposition} 
 \begin{proof}
 	Choose coordinates on $\bbP^2$ such that the three coordinate vertices lie on smooth points of $C$, and each coordinate line meets $C$ in exactly $d=\deg(C)$ distinct points.
 	Consider the Kummer cover $f([x:y:z])=[x^k:y^k:z^k]$, and let $\mathcal{K}=\MSing(C)$ as before. Then
\begin{equation*}
h\le h(f^*(C))=\frac{k^2 d^2-k^2 \mathcal{K}^2-3k^2}{k^2|K|+3}=
h(C)\frac{k^2|K|}{k^2 |K| +3}-\frac{3k^2}{k^2|K|+3}.
\end{equation*} 	
The limit of the right hand side for $k\to \infty$ is $h(C)-3/|K|$, whence the claim.
 \end{proof}

\subsection{Examples}
\label{sec:examples}

\subsubsection*{Fermat arrangements}
We observe that some of the known curve arrangements with negative $H$-indices can be obtained as pullbacks of simpler arrangements by suitable ramified maps.

Let $C\subset \bbP^2$ be a reducible cubic made up of three concurrent lines; for simplicity assume it is given by the homogeneous equation $(x-y)(y-z)(z-x)=0$.  
Obviously $\MSing(C)=\{p\}$ is the single point $p=[1:1:1]$ with multiplicity 3, and $h(C)=0$.
Let $f_k:\bbP^2\rightarrow\bbP^2$ be the Kummer cover $f_k([x:y:z])=[x^k:y^k:z^k]$. 
The so-called \emph{$k$-th Fermat arrangement of lines} is the reduced curve $f_k^*(C)$, which has $k^2$ triple points, three points of multiplicity $k$, and computing as in \eqref{eq:h-pullback} we obtain $h(f_k^*(C))=-3k^2/(k^2+3)$.


\subsubsection*{Curves with ordinary singularities}

	In the proof of Propositions \ref{pro:infimum} and \ref{pro:h-bound} we used morphisms which have multiplicity $>1$ at smooth points of a given curve $C$, for simplicity.
	In practice, in the search for curves with very negative $H$-indices, using a morphism with multiplicity $>1$ at singular points of $C$ turns out to be more effective.

For instance, if we apply the strategy of Proposition \ref{pro:submultiplicative} to the Wiman configuration $W$ of 45 lines with 201 singular points and $h$-index $-225/67 \simeq -3.358$ we can deduce the existence of curves with ordinary singularities and $h$-index arbitrarily close to $-225/67-3/201=-226/67\simeq -3.373$.

However, we can obtain a more negative index. Consider a projective coordinate system which has its three coordinate points sitting in triple points of the Wiman configuration $W$ of 45 lines \cite{BdRHHLPSz} and such that none of the coordinate lines belong to $W$.
Then the intersection of each coordinate line with $W$ consists of the two chosen triple points and 39 transverse intersections with the lines not going through the triple points (this is presumably well known, we checked it using Singular).
Denote as before $\mathcal{K}=\WSing(W)$, and apply the Kummer cover $f([x:y:z])=[x^k:y^k:z^k]$ to $W$. 
Each vertex $p$ of the coordinate system is its unique preimage, and $f^*(W)$ has an ordinary singularity of multiplicity $k^2 \mult_{p}(W)$ there, so:
\[
h(f^*(W))=\frac{45^2\,k^2 -k^2 \mathcal{K}^2}{k^2(|K|-3)+3}=
h(C)\frac{k^2|K|}{k^2 (|K|-3) +3}=-\frac{225}{67}\cdot\frac{201\,k^2}{198\,k^2+3}.
\]
By taking large values of $k$, we see that \emph{there exist reduced curves $C\subset \bbP^2$ with ordinary singularities and Harbourne index arbitrarily close to} 
\[-\frac{225}{67}\cdot\frac{201}{198}=-\frac{25}{7}\simeq -3.571.\] 
This proves Theorem \ref{intro:newrecord}.
\subsubsection*{Klein-invariant configurations of higher degree}

In \cite{PR}, we described the singularities of the configuration of 21 reducible polars to the Klein quartic 
$$\Phi_{4}: x^{3}y + y^{3}z +z^{3}x,$$
computing in particular their $H$-constants, and we introduced additional very negative configurations of curves of higher degree. 
We next recall the construction and give an explicit description of some clusters of singular points of these configurations, leading to a bound on their $h$-indices.

Denote $f :\bbP^2\rightarrow \bbP^2$ the gradient map given 
by the partial derivatives of Klein's quartic equation, explicitly
$$\mathbb{P}^{2} \ni [x:y:z] \overset{9:1}{\longmapsto} [u:v:w]=\left[3x^{2}y + z^{3}: 3y^{2}z + x^{3}: 3z^{2}x+y^{3} \right] \in \mathbb{P}^{2}.$$
Let also $\Phi_{21}$ be the polynomial of degree $21$, invariant under the group $G_{168}$ of projectivities fixing $\Phi_{4}$, which defines the so-called Klein configuration $K=K_0$ of 21 lines.
It was showed in \cite[Proposition 2.1]{PR} that $\Phi_{63}=f^*(\Phi_{21})$ defines the configuration $K_1$ of 21 reducible polars of Klein's quartic, which splits as $\Phi_{63}=\Phi_{42}\Phi_{21}$, where $\Phi_{42}=0$ is a configuration of 21 irreducible conics.
Iterating the process, let $\Phi_{189}=f^*(\Phi_{63})$, splitting as $\Phi_{189}=\Phi_{126}\Phi_{42}\Phi_{21}$, where $\Phi_{126}=f^*(\Phi_{42})$ is an invariant configuration of $21$ sextics, and in general for $k\ge 1$,
\[
(f^k)^*(\Phi_{21})=\Phi_{14\cdot 3^{k}}\cdots 
\Phi_{42}\Phi_{21}.
\]
We shall not attempt at a complete description of the singularities of the configurations $K_k: (f^k)^*(\Phi_{21})=0$, but we focus on the singularities lying on the preimage of the singular points of $K_2:\Phi_{63}=0$; these are enough to show that the Harbourne index $h_k$ of $K_k$ is a decreasing sequence whose limit is at most $-1283/410\simeq-3.123$.

\begin{proposition} \label{pro:singularities_polars}
The singularities of the arrangement $K_2$ are $42$ nodes, $252$ ordinary triple points and $189$ ordinary quadruple points.
\end{proposition}

This can be equivalently stated as follows:

\begin{proposition} \label{pro:cluster_polars}
	The multi-cluster $\WSing(K_1)$ of singularities of the arrangement of reducible polars of Klein's quartic consists of $483$ clusters of one point each, of which $42$ have multiplicity $2$, $252$ have multiplicity $3$, and $189$ have multiplicity $4$.
\end{proposition}

Proposition \ref{pro:singularities_polars} and Lemmas \ref{lem:tangency}, \ref{lem:reduced} below were proven in \cite{PR}.
We denote $\Orb_{42}$ the set of 42 nodes of $K_2$ (which is the unique orbit of size 42 for the Klein group $G_{168}$). 

\begin{lemma}\label{lem:tangency}
For every $k\ge 1$, $\Phi_{14\cdot 3^{k}}$ is smooth along $\Orb_{42}$.
Moreover, the local intersection multiplicity at $p\in \Orb_{42}$
of $\Phi_{14\cdot 3^{k}}$ and $\Phi_{14\cdot 3^{k+1}}$ is $2\cdot 3^{k-1}$.
\end{lemma}

\begin{lemma}\label{lem:reduced}
	For every $k\ge 1$, $K_k$ is reduced.
\end{lemma}

\begin{proposition}\label{pro:sing42}
	For every $k\ge 2$, the cluster $\mathcal{S}_{p.k}$ of singular points at $p\in \Orb_{42}$ of $(f^k)^*(\Phi_{21})$ consists of $2\cdot 3^{k-2}$ points, totally ordered by infinitely-near-ness, of which 
	\begin{enumerate}
		\item $p$ has multiplicity $k+1$,
		\item the first point infinitely near to $p$ has multiplicity $k$,
		\item for every $m=2,\dots,k-1$, there are exactly $4\cdot 3^{k-m-1}$ points of multiplicity $m$.
	\end{enumerate}
\end{proposition}
\begin{proof}
	Let $L$ be the line, component of the Klein configuration $K_0$, through $p$, and let $C$ be the conic, component of $K_1$ through $p$.
	By Lemma \ref{lem:tangency}, the singularity of $(f^k)^*(\Phi_{21})$ at $p$ consists of $k+1$ smooth branches, of which all but $L$ are tangent. This proves the first two items. 
	Also by Lemma \ref{lem:tangency}, if $k\ge 3$ the $m$ branches of $(f^k)^*(\Phi_{21})$ corresponding to the curves $(f^{k-1})^*(C), (f^{k-2})^*(C), \dots (f^{k-m+1})^*(C)$ share their first $2\cdot 3^{k-m}$ points infinitely near to $p$ (including $p$); in particular all singular points at $p\in \Orb_{42}$ of $(f^k)^*(\Phi_{21})$ belong to the smooth branch of $(f^{k-1})^*(C)$ through $p$ and hence they are totally ordered by infinitely-near-ness. The third item follows by observing that $4\cdot 3^{k-m-1}=2\cdot 3^{k-m}-2\cdot 3^{k-(m+1)}$.
\end{proof}


%
%
Next, we will define iteratively a weighted multi-cluster which will give our upper bound for the $h$-index of the configurations $K_k$. Let $\mathcal{K}_1=\WSing(K_1)$.
By Proposition \ref{pro:cluster_polars}, we can split
\[\mathcal{K}_1= \mathcal{S} \cup \mathcal{T}  \]
where $\mathcal{S}$ consists of the points in $\Orb_{42}$ with multiplicity 2, and $\mathcal{T}$ consists of 252 points of multiplicity 3 and 189 of multiplicity 4.
For every $p\in \Orb_{42}$, denoting again by $L$ and $C$ the line and conic components of $K_1$ through $p$, Lemma \ref{lem:tangency} shows that the local intersection multiplicity at $p$ of $f^*L$ and $f^*C$ is $2+1=3$, so $f(\Orb_{42})=\Orb_{42}$ and $f$ has local degree 3 at every point of $\Orb_{42}$. Therefore we can write
\[f^{-1}(\Orb_{42})=\Orb_{42} \cup X\]
where $X$ is a finite set of points with $\sum_{p \in {X}} \deg_p f = 6\cdot 42$. 
Denote, for each $k\ge 2$, $\mathcal{S}_k=\cup_{p\in \Orb_{42}}\mathcal{S}_{p,k}$ the multi-cluster of singular points of $(f^k)^*(\Phi_{21})$ supported at $\Orb_{42}$, and split its pullback as
\[f^{*}(\mathcal{S}_k)=\mathcal{S}_k^{42} \cup \mathcal{S}_k^X\]
where $\mathcal{S}_k^{42}$ is the subcluster supported at $\Orb_{42}$ and $\mathcal{S}_k^X$ is the subcluster supported at $X$.

Finally, define
\[
\mathcal{K}_k=\mathcal{S}_k \cup \mathcal{S}_{k-1}^X \cup f^*(\mathcal{S}_{k-2}^X)
\cup \dots \cup (f^{k-2})^*(\mathcal{S}_{1}^X) \cup (f^{k-1})^*(\mathcal{T})
\]

\begin{proposition}
For every $k\ge 1$, the arrangement $K_{k}$ goes through the weighted multi-cluster $\mathcal{K}_k$.
Moreover, for every $k \geq 2$, the self-intersection and cardinality of $\mathcal{K}_k$ satisfies: 
\[
\mathcal{K}_k^2=\frac{21}{2}(53\cdot 9^{k}+3)-196\cdot 3^{k+1}, \quad
|\mathcal{K}_k|\le 84\cdot 9^{k}-28\cdot 3^{k+1}.
\]
\end{proposition}

\begin{proof}
	The fact that $K_{k}$ goes through $\mathcal{K}_k$ is clear by the construction of $\mathcal{K}_k$.

By Corollary \ref{cor:multiplicative} and Proposition \ref{pro:submultiplicative},
\begin{equation}\label{eq:compT}
((f^{k-1})^*\mathcal{T})^2=9^{k-1}\cdot \mathcal{T}^2=9^{k}\cdot 588, \quad
|(f^{k-1})^*\mathcal{T}| \le 9^{k-1}\cdot |\mathcal{T}| = 9^{k} \cdot 49.
\end{equation}
On the other hand, by Proposition \ref{pro:sing42}, for $k\ge 2$,
\begin{gather}
\mathcal{S}_k^2=42\left((k+1)^2+k^2+4\cdot\sum_{m=2}^{k-1}3^{k-m-1}m^2\right)=588\cdot 3^{k-2}-42, \\ |\mathcal{S}_k|= 42\left(2+\sum_{m=2}^{k-1}3^{k-m-1}\right)=84 \cdot 3^{k-2},
\end{gather}
and therefore, for $\ell=2,\dots,k-1$, we obtain applying Corollary \ref{cor:multiplicative} and Proposition \ref{pro:submultiplicative} again,
\begin{gather}
\left((f^{k-\ell-1})^*(\mathcal{S}_\ell^X)\right)^2=1176\cdot 3^{k+\ell-4} - 28\cdot 9^{\ell-1}, \\
\left|(f^{k-\ell-1})^*(\mathcal{S}_\ell^X)\right|\le 168\cdot 3^{k+\ell-4},
\end{gather}
and, finally, since $(\mathcal{S}_1^X)^2=6\cdot42 \cdot 4$ and $|\mathcal{S}_1^X|\le 6 \cdot 42$,
\begin{gather}
\left((f^{k-2})^*(\mathcal{S}_1^X)\right)^2=112\cdot 9^{k-1}, \\ \label{eq:compS}
\left|(f^{k-2})^*(\mathcal{S}_1^X)\right|\le 28\cdot 9^{k-1}.
\end{gather}
Summing up \eqref{eq:compT}--\eqref{eq:compS}, we obtain the claim.
\end{proof}

\begin{corollary}
	For all $k\ge 2$, 
\[
h(K_k) \le
-\frac{1283\cdot 9^k-81}{410\cdot 9^k} \underset{k\to\infty}{\longrightarrow}-\frac{1283}{410}\simeq -3.12927.
\]
\end{corollary}
\section*{Acknowledgments}
This work was begun during the first author's visit at Universitat Aut\`{o}\-no\-ma de Barcelona, under the financial support of the Spanish MINECO grant MTM2016-75980-P, which also supports the second author. Finally, we would like to warmly thank an anonymous referee for useful comments that allowed to improve the paper.
During the project the first author was supported by the Fundation for Polish Science (FNP) Scholarship Start No. 076/2018.

{\footnotesize
	\bibliographystyle{plainurl}
\bibliography{Ramified}{}}

\bigskip
   Piotr Pokora,
   Institute of Mathematics,
   Polish Academy of Sciences,
   ul. \'{S}niadeckich 8,
   PL-00-656 Warszawa, Poland \\
\nopagebreak
   \textit{E-mail address:} \texttt{piotrpkr@gmail.com, ppokora@impan.pl}
   
\bigskip
	Joaquim Ro\'{e},
	Universitat Aut\`{o}noma de Barcelona, Departament de Matem\`{a}tiques,
08193 Bellaterra (Barcelona), Spain. \\
\nopagebreak
\textit{E-mail address:} \texttt{jroe@mat.uab.cat}
\end{document}